\documentclass[12pt,a4paper,reqno]{article}

\usepackage{ifxetex}
\ifxetex
	\usepackage{csquotes}
	\usepackage{polyglossia}
	\setdefaultlanguage{english}
\else 
	\usepackage[utf8]{inputenc}
	\usepackage[T1]{fontenc}
	\usepackage{csquotes}
	\usepackage[english]{babel}
\fi

\usepackage{microtype}
\usepackage{amsmath}
\usepackage{amssymb}
\usepackage{amsthm}
\usepackage{amsfonts}
\usepackage{verbatim}
\usepackage{enumitem}
\usepackage{url}
\usepackage{hyperref}
\usepackage{mathrsfs}
\usepackage{bbm}
\usepackage{multirow}
\usepackage{tikz}

\usetikzlibrary{arrows}
\usetikzlibrary[patterns]
\usetikzlibrary{decorations.pathreplacing}
\usetikzlibrary{shapes.misc}

\tikzset{cross/.style={cross out, draw=black, minimum size=2*(#1-\pgflinewidth), inner sep=0pt, outer sep=0pt},
cross/.default={1pt}}

\definecolor{ffqqqq}{rgb}{1,0,0}
\definecolor{qqffqq}{rgb}{0,1,0}
\definecolor{ffffff}{rgb}{1,1,1}
\colorlet{ColorGray}{gray!30}

\usepackage[numeric,initials,nobysame]{amsrefs}
\usepackage{verbatim}
\usepackage{subcaption}
\usepackage{color}
\usepackage[normalem]{ulem}
\usepackage{tikz}
\usetikzlibrary{patterns,calc}
\usepackage{pgfplots}
\pgfplotsset{compat=newest}

\newtheorem{thm}{Theorem}

\newtheorem{lem}[thm]{Lemma}
\newtheorem{prop}[thm]{Proposition}

\newtheorem{ques}{Question}
\theoremstyle{definition}
\newtheorem{defn}[thm]{Definition}
\newtheorem{rem}[thm]{Remark}

\newtheorem{obs}[thm]{Observation}

\numberwithin{thm}{section}

\def\cC{\mathcal{C}}

\def\cO{\mathcal{O}}
\def\cP{\mathcal{P}}

\def\cS{\mathcal{S}}
\def\cT{\mathcal{T}}
\def\cU{\mathcal{U}}

\renewcommand{\H}{\mathbb{H}}

\def\N{\mathbb{N}}
\def\Pr{\mathbb{P}}

\def\R{\mathbb{R}}

\def\Z{\mathbb{Z}}

\renewcommand{\le}{\leqslant}
\renewcommand{\ge}{\geqslant}
\renewcommand{\to}{\rightarrow}

\def\<{\langle}
\def\>{\rangle}

\title{Complexity of 2D bootstrap percolation difficulty: Algorithm and NP-hardness}
\date{\today}
\author{Ivailo Hartarsky\thanks{DMA, CNRS, UMR 8553, \'Ecole Normale Sup\'erieure, PSL University, Paris.\newline Present address: CEREMADE, CNRS, UMR 7534, Universit\'e Paris-Dauphine, PSL University, Place du Mar\'echal de Lattre de Tassigny, 75016, Paris, France. The author was supported in part by European Research Council Starting Grant 680275 MALIG.\newline\textsf{hartarsky@ceremade.dauphine.fr}} \and Tam\'as R\'obert Mezei\thanks{Alfréd Rényi Institute of Mathematics, 13--15 Re\'altanoda utca, 1053 Budapest, Hungary. The author was supported in part by the National Research, Development and Innovation Office (NKFIH) grant K-116769 and KH-126853. \newline\textsf{mezei.tamas.robert@renyi.hu}}}

\begin{document}
\maketitle

\begin{abstract}
Bootstrap percolation is a class of cellular automata with random initial state.
Two-dimensional bootstrap percolation models have three rough universality
classes, the most studied being the ``critical'' one. For this class the scaling
of the quantity of greatest interest (the critical probability) was determined by Bollobás, Duminil-Copin, Morris and Smith~\cite{Bollobas14} in terms of a simply defined combinatorial quantity called ``difficulty'', so the subject seemed closed up to finding sharper results. However, the computation of the difficulty was never considered. In this paper we provide the first algorithm to determine this quantity, which is, surprisingly, not as easy as the definition leads to thinking. The proof also provides some explicit upper bounds, which are of use for bootstrap percolation. On the other hand, we also prove the negative result that computing the difficulty of a critical model is NP-hard. This two-dimensional picture contrasts with an upcoming result of Balister, Bollobás, Morris and Smith~\cite{Balister18b} on uncomputability in higher dimensions. The proof of NP-hardness is achieved by a technical reduction to the \textsc{Set Cover} problem.
\end{abstract}

\section{Introduction}\label{sec:intro}
\subsection{Background}
Bootstrap percolation is a class of cellular automata whose first representative was introduced in 1979 by Chalupa, Leath and Reich~\cite{Chalupa79} in statistical physics. Further applications to several other areas have been considered, namely dynamics of the Ising model, kinetically constrained models for the glass transition, abelian sandpiles and others (see a recent review of Morris~\cite{Morris17} for more information).

We consider the following iterative discrete-time process on the elements (\emph{sites}) of $\Z^d$. At each time $t\in\N$ every site is either \emph{infected} or \emph{healthy}. We encode the state of all sites by specifying the set of infected sites $A_t$. Given a set $A\subseteq\Z^d$ or ${(\Z/n\Z)}^d$ of initially infected sites, more sites become infected at each discrete time step following a deterministic monotone local rule invariant in time and space, while infections never heal. More precisely, let us introduce the broadest framework brought forward by Bollobás, Smith and Uzzell~\cite{Bollobas15}.\footnote{Earlier partly non-rigorous considerations of a more restricted class of models can be found in the works of Gravner and Griffeath~\cites{Gravner96,Gravner99} from the 1990s.}

A bootstrap percolation model is specified by a finite set $\cU$, called the \emph{update family}, of finite subsets of $\Z^d\setminus\{0\}$, called \emph{rules}. For an initial set of \emph{infected} sites $A=A_0\subseteq \Z^d$ we recursively define for all $t\in\N$
\[A_{t+1}=A_t\cup\{x\in \Z^d\;:\;\exists U\in\cU,\;x+U\subseteq A_t\}\]
and $[A]=\bigcup_{t\ge0}A_t$ is the \emph{closure} of $A$ with respect to this operation.

For concreteness, four examples of such models with different update families $\cU$ are given in Figure~\ref{fig:example}. We will use those to also illustrate further definitions. For instance, in the East model (see Figure~\ref{fig:East}) one infects sites whose bottom or left neighbour is infected, while in the North-East model (Figure~\ref{fig:NE}) one only infects sites such that both their bottom and left neighbours are infected.
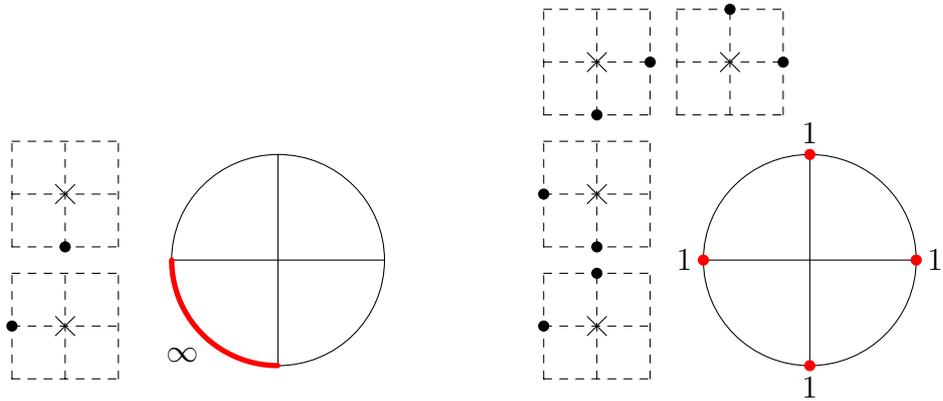
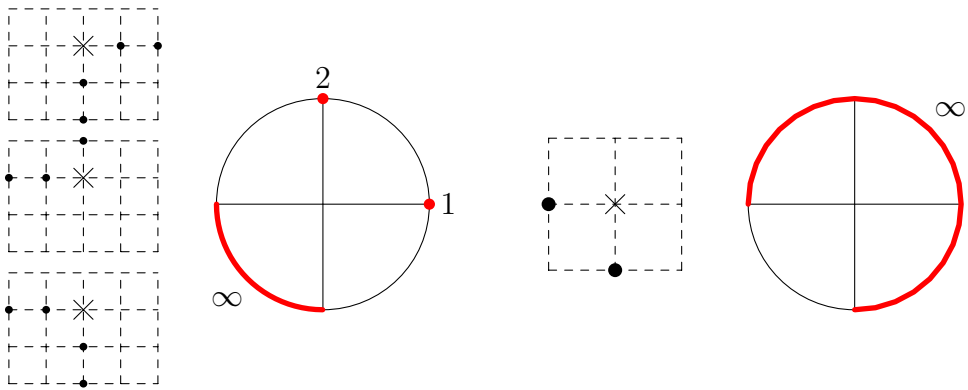
\begin{figure}
	\centering
	\begin{subfigure}{0.45\textwidth}
		\centering
		\begin{tikzpicture}[line cap=round,line join=round,x=1.0cm,y=1.0cm, scale=0.7]
			\clip (-1.5,-1.5) rectangle (6.5,6.5);
			\begin{scope}[shift={(0,0)}]
				\draw [dash pattern=on 3pt off 3pt, xstep=1cm,ystep=1cm] (-1,-1) grid (1,1);
				\fill (-1,0) circle (3pt);
				\draw (0,0) node[cross=4pt,rotate=0] {};
			\end{scope}
			\begin{scope}[shift={(0,2.5)}]
				\draw [dash pattern=on 3pt off 3pt, xstep=1cm,ystep=1cm] (-1,-1) grid (1,1);
				\fill (0,-1) circle (3pt);
				\draw (0,0) node[cross=4pt,rotate=0] {};
			\end{scope}
			\begin{scope}[shift={(4,1.25)},x=2.0cm,y=2.0cm]
				\draw(0,0) circle (2cm);
				\draw (0,0)-- (1,0);
				\draw (0,1)-- (0,0);
				\draw (0,0)-- (-1,0);
				\draw (0,0)-- (0,-1);
				\draw [shift={(0,0)},line width=2pt,color=ffqqqq]
					plot[domain=pi:4.71,variable=\t]
					({1*1*cos(\t r)+0*1*sin(\t r)},{0*1*cos(\t r)+1*1*sin(\t r)});
				\draw (-0.9,-0.9) node {$\infty$};
			\end{scope}
		\end{tikzpicture}
		\caption{\label{fig:East}The East model, which is supercritical (with difficulty $0$).}
	\end{subfigure}\qquad%
	\begin{subfigure}{0.45\textwidth}
		\centering
		\begin{tikzpicture}[line cap=round,line join=round,x=1.0cm,y=1.0cm, scale=0.7]
			\clip (-1.5,-1.5) rectangle (6.5,6.5);
			\begin{scope}[shift={(0,0)}]
				\draw [dash pattern=on 3pt off 3pt, xstep=1cm,ystep=1cm] (-1,-1) grid (1,1);
				\fill (-1,0) circle (3pt);
				\fill (0,1) circle (3pt);
				\draw (0,0) node[cross=4pt,rotate=0] {};
			\end{scope}

			\begin{scope}[shift={(0,2.5)}]
				\draw [dash pattern=on 3pt off 3pt, xstep=1cm,ystep=1cm] (-1,-1) grid (1,1);
				\fill (0,-1) circle (3pt);
				\fill (-1,0) circle (3pt);
				\draw (0,0) node[cross=4pt,rotate=0] {};
			\end{scope}

			\begin{scope}[shift={(0,5)}]
				\draw [dash pattern=on 3pt off 3pt, xstep=1cm,ystep=1cm] (-1,-1) grid (1,1);
				\fill (1,0) circle (3pt);
				\fill (0,-1) circle (3pt);
				\draw (0,0) node[cross=4pt,rotate=0] {};
			\end{scope}

			\begin{scope}[shift={(2.5,5)}]
				\draw [dash pattern=on 3pt off 3pt, xstep=1cm,ystep=1cm] (-1,-1) grid (1,1);
				\fill (0,1) circle (3pt);
				\fill (1,0) circle (3pt);
				\draw (0,0) node[cross=4pt,rotate=0] {};
			\end{scope}

			\begin{scope}[shift={(4,1.25)},x=2.0cm,y=2.0cm]
				\draw(0,0) circle (2cm);
				\draw (0,0)-- (1,0);
				\draw (0,1)-- (0,0);
				\draw (0,0)-- (-1,0);
				\draw (0,0)-- (0,-1);
				\fill [color=ffqqqq] (1,0) circle (3pt) node[anchor=west,black] {$1$};
				\fill [color=ffqqqq] (0,1) circle (3pt) node[anchor=south,black] {$1$};
				\fill [color=ffqqqq] (-1,0) circle (3pt) node[anchor=east,black] {$1$};
				\fill [color=ffqqqq] (0,-1) circle (3pt) node[anchor=north,black] {$1$};
			\end{scope}
		\end{tikzpicture}
		\caption{\label{fig:2neighbour}The modified $2$-neighbour model, which is
		critical with difficulty $1$.}
	\end{subfigure}
	\\
	\begin{subfigure}{0.45\textwidth}
		\centering
		\begin{tikzpicture}[line cap=round,line join=round,x=1.0cm,y=1.0cm, scale=0.7]
				\clip (-1.75,-1.75) rectangle (7.0,6.5);
				\begin{scope}[shift={(0,0)},scale=0.7]
				\draw [dash pattern=on 3pt off 3pt, xstep=1cm,ystep=1cm] (-2,-2) grid (2,1);
				\fill (-1,0) circle (3pt);
				\fill (-2,0) circle (3pt);
				\fill (0,-1) circle (3pt);
				\fill (0,-2) circle (3pt);
				\draw (0,0) node[cross=4pt,rotate=0] {};
			\end{scope}
			\begin{scope}[shift={(0,2.5)},scale=0.7]
				\draw [dash pattern=on 3pt off 3pt, xstep=1cm,ystep=1cm] (-2,-2) grid (2,1);
				\fill (-1,0) circle (3pt);
				\fill (-2,0) circle (3pt);
				\fill (0,1) circle (3pt);
				\draw (0,0) node[cross=4pt,rotate=0] {};
			\end{scope}
			\begin{scope}[shift={(0,5)},scale=0.7]
				\draw [dash pattern=on 3pt off 3pt, xstep=1cm,ystep=1cm] (-2,-2) grid (2,1);
				\fill (1,0) circle (3pt);
				\fill (2,0) circle (3pt);
				\fill (0,-1) circle (3pt);
				\fill (0,-2) circle (3pt);
				\draw (0,0) node[cross=4pt,rotate=0] {};
			\end{scope}
			\begin{scope}[shift={(4.5,2)},x=2.0cm,y=2.0cm]
			\draw(0,0) circle (2cm);
			\draw (0,0)-- (1,0);
			\draw (0,1)-- (0,0);
			\draw (0,0)-- (-1,0);
			\draw (0,0)-- (0,-1);
			\fill [color=ffqqqq] (1,0) circle (3pt) node[anchor=west,black] {$1$};
			\fill [color=ffqqqq] (0,1) circle (3pt) node[anchor=south,black] {$2$};
			\draw [shift={(0,0)},line width=2pt,color=ffqqqq] plot[domain=pi:4.71,variable=\t]({1*1*cos(\t r)+0*1*sin(\t r)},{0*1*cos(\t r)+1*1*sin(\t r)});
			\draw (-0.9,-0.9) node {$\infty$};
		\end{scope}
	\end{tikzpicture}
	\caption{\label{fig:toy}A toy model, which is critical with difficulty $1$.}
\end{subfigure}\qquad%
	\begin{subfigure}{0.45\textwidth}
		\centering
		\begin{tikzpicture}[line cap=round,line join=round,x=1.0cm,y=1.0cm,scale=0.7]
			\clip (-1.75,-1.75) rectangle (7.0,6.5);
			\begin{scope}[shift={(0,2)},scale=1.25]
				\draw [dash pattern=on 3pt off 3pt, xstep=1cm,ystep=1cm] (-1,-1) grid (1,1);
				\fill (-1,0) circle (3pt);
				\fill (0,-1) circle (3pt);
				\draw (0,0) node[cross=4pt,rotate=0] {};
			\end{scope}
			\begin{scope}[shift={(4.5,2)},x=2.0cm,y=2.0cm]
				\draw(0,0) circle (2cm);
				\draw (0,0)-- (1,0);
				\draw (0,1)-- (0,0);
				\draw (0,0)-- (-1,0);
				\draw (0,0)-- (0,-1);
				\draw [shift={(0,0)},line width=2pt,color=ffqqqq] plot[domain=-pi/2:pi,variable=\t]({1*1*cos(\t r)+0*1*sin(\t r)},{0*1*cos(\t r)+1*1*sin(\t r)});
				\draw (0.9,0.9) node {$\infty$};
			\end{scope}
		\end{tikzpicture}
		\caption{The North-East model, which is subcritical (with difficulty $\infty$).}\label{fig:NE}
	\end{subfigure}
	\caption{Four example bootstrap percolation models. For each one the rules are depicted on the left with $0$ marked by a cross, the sites of each rule denoted by dots and the grid lines dashed. The figure on the right gives the stable directions in red with their difficulties next to them. The isolated stable directions are marked by red dots.}\label{fig:example}
\end{figure}

We will only discuss the most studied case, where $A$ is chosen at random according to the product Bernoulli measure $\Pr_p$, so that each site is initially infected with probability $p\in[0,1]$. Equipped with this measure, the model exhibits a phase transition at
\[p_c=\inf\{p\in[0,1]\;:\;\Pr_p(0\in[A])=1\}.\]
The model is defined identically on tori ${(\Z/n\Z)}^d$ by setting
\[p_c(n)=\inf\{p\in[0,1]\;:\;\Pr_p([A]={(\Z/n\Z)}^d)\ge 1/2\}.\]
In this background section we consider $n\to 0$ and use associated asymptotic notation. Namely, given a function $f(n)$ we write $\cO(f(n))$ for a function bounded in absolute value by $Cf(n)$ for some constant $C$ not depending on $n$ (but possibly depending on $\cU$). We write $\Theta(f(n))$ for a function that is bounded above by $Cf(n)$ and below by $cf(n)$ for some positive constants $c$ and $C$, neither depending on $n$. We will use analogous notation in later sections with respect to other diverging parameters.

Although for some concrete models higher dimensions have been understood and some general universality conjectures have been put forward in~\cite{Balister16}*{Conjecture~16} and~\cite{Bollobas14}*{Conjecture~9.2}, we will restrict our attention to the 2-dimensional case. The results of Bollobás, Smith and Uzzell~\cite{Bollobas15} and Balister, Bollobás, Przykucki and Smith~\cite{Balister16} combined establish that all bootstrap percolation models can be partitioned (by a simple procedure) into $3$ ``rough universality classes'' with qualitatively different scaling of $p_c(n)$. In order to define these we need some notation. For a direction $u$ in the unit circle $S^1=\{x\in\R^2\;:\;\|x\|_2=1\}$, which we standardly identify with $\R/2\pi\Z$, we denote by
\[\H_u=\{x\in\Z^2\;:\;\<x,u\><0\}\]
the open half-plane with normal $u$ and by
\[l_u=\{x\in\Z^2,\;\<x,u\>=0\}\]
the line passing through $0$ perpendicular to $u$. A direction $u$ is \emph{unstable} if there exists $U\in\cU$ such that $U\subset\H_u$ and \emph{stable} otherwise. It is not difficult to show that the unstable directions form a finite union of open intervals in $S^1$ with \emph{rational} endpoints, that is a direction $u$ such that $l_u\cap\Z^2\neq \varnothing$. Indeed, each rule individually induces a (possibly empty) interval of unstable directions with endpoints perpendicular to sites in the rule (so in $\Z^2$), there are finitely many rules and, by definition, the union of these intervals is the set of unstable directions for the full model. Thus, the set of stable directions is a finite union of closed intervals with rational endpoints in $S^1$, some of which may be reduced to a single point called \emph{isolated stable direction}.

As an example, let us consider the modified $2$-neighbour model (Figure~\ref{fig:2neighbour}). The top-left rule consisting of $(1,0)\in\Z^2$ and $(0,-1)\in\Z^2$ makes all directions in the open interval $(\pi/2,\pi)\subset S^1$ unstable. By invariance by rotation by $\pi/2$ there remain only the four isolated stable directions shown in Figure~\ref{fig:2neighbour}. The reader is encouraged to check the stable directions of the other examples in Figure~\ref{fig:example}.

We are now ready to define the partition into rough universality classes conjectured in~\cite{Bollobas15} and proved in~\cites{Bollobas15,Balister16} is in terms of these directions.
\begin{itemize}
\item $\cU$ is \emph{supercritical} if there exists an open semi-circle of unstable directions, in which case $p_c(n)=n^{-\Theta(1)}$.
\item $\cU$ is \emph{critical} if it is not supercritical and there exists a semi-circle with a finite number of stable directions, in which case $p_c(n)={(\log n)}^{-\Theta(1)}$.
\item $\cU$ is \emph{subcritical} otherwise (if each semi-circle contains infinitely many stable directions), in which case $p_c>0$.
\end{itemize}
Let us check that the modified $2$-neighbour model (Figure~\ref{fig:2neighbour})
is critical. As observed before, the only stable directions are the four axis
directions. In particular, every open semi-circle contains either one or two of
them. For the toy model (Figure~\ref{fig:toy}) again every open semi-circle
contains at least one of the stable directions, but e.g.\ the semi-circle $(-\pi/2,\pi/2)\subset S^1$ only contains one stable direction, so it is also critical. For the East model the same semi-circle contains no stable directions, making it supercritical. Finally, in the North-East model there is only a single quarter of a circle of unstable directions. In particular, every half-circle contains infinitely many unstable directions, so the model is subcritical.

The behavior of supercritical models is dominated by the study of finite infected sets with infinite closure (a single infected site in the East model), while subcritical ones are more closely related to percolation (for example, the North-East model is equivalent to classical oriented site percolation if one considers healthy sites). The most studied models are critical ones, to which the archetypal example of bootstrap percolation belongs --- the $2$-neighbor model, in which a site becomes infected if at least two of its nearest neighbors are already infected. Note that the modified $2$-neighbour model in Figure~\ref{fig:2neighbour} does not infect a site if it only has $2$ infected neighbours which are on opposite sides of it, however, from the point of view of stable directions and difficulties to be defined later, this modification is of no importance. The $2$-neighbour model is the first one for which the rough universality result above (and more) was established --- by Aizenman and Lebowitz~\cite{Aizenman88}. They realized that the dynamics is dominated by a bottleneck --- creating an infected ``droplet'' of a certain ``critical'' size, which can then easily grow out to infinity, and proved that for this model $p_c(n)=\Theta(1/\log n)$. In a substantial breakthrough Holroyd~\cite{Holroyd03} determined the asymptotic location of the sharp threshold and since then much sharper results have been proved~\cites{Gravner08,Hartarsky19Rob}:
\[p_c=\frac{\pi^2}{18\log n}-\frac{\Theta(1)}{{(\log n)}^{3/2}}.\]

Such sharp or sharper bounds have been obtained for a handful of other specific models~\cites{Duminil18b,Bollobas17,Duminil12}, but still remain open in general. However, the level of precision of the Aizenman-Lebowitz result was established in full generality for critical models by Bollobás, Duminil-Copin, Morris and Smith~\cite{Bollobas14}. They introduce the following key notion of \emph{difficulty}.
\begin{defn}[Definition~1.2 of~\cite{Bollobas14}]\label{def:diff}
Let $\cU$ be a critical model and $u$ be a direction. If $u$ is an isolated
stable direction, we define its \emph{difficulty}, $\alpha(u)$, to be the
minimum cardinality of a set $Z\subseteq \Z^2\setminus\H_u$ such that $\bar Z:=[\H_u\cup Z]\setminus\H_u$ is infinite. For unstable directions $u$ we set $\alpha(u)=0$ and for non-isolated stable ones we set $\alpha(u)=\infty$. The \emph{difficulty} of $\cU$ is
\begin{equation}
\label{eq:alpha:def}
\alpha=\inf_{C\in\cC}\sup_{u\in C}\alpha(u),
\end{equation}
where $\cC$ is the set of open semi-circles of $S^1$.
\end{defn}
Let us note that the definition we give is formally different from the one in~\cite{Bollobas14}, but it turns out to be equivalent. Indeed, any unstable direction $u$ satisfies $[\H_u]=\Z^2$, since one can infect $0$ by definition of unstable directions and, by translation invariance one can infect $l_u$, so that a translate of $\H_u$ becomes infected and one may conclude by induction. Here we used that for any rational direction, we can write $\Z^2=\bigsqcup_{i\in\Z} (l_u+i\cdot x_u)$ for some vector $x_u\in\Z^2$, where we write $A+x$ for $\{a+x,a\in A\}$ for any set $A\subseteq\Z^2$ and site $x\in\Z^2$. For stable directions the equivalence is proved in Lemma~2.7 of~\cite{Bollobas14}.

For the reader's convenience, let us determine the difficulties of the stable directions of the toy model of Figure~\ref{fig:toy}. By definition unstable directions have difficulty $0$ and non-isolated stable ones have difficulty $\infty$, so we are left with the right ($0$) and top ($\pi/2$) isolated stable directions. Let us start with the direction $0$. Since it is stable $[\H_0]=\H_0$, we have $\alpha\ge 1$.\footnote{More generally, for any model and any isolated stable direction $u$ we have $1\le \alpha(u)<\infty$ (see Lemma~2.7 of~\cite{Bollobas14}).} However, $[\H_0\cup \{(0,0)\}]=\H_0\cup l_0$, since one can infect $(0,-1)$ by the second rule (see Figure~\ref{fig:toy}) and, inductively $(0,-k)$ for all $k\in\N$; one can also use the first rule to infect $(0,1)$ and then $(0,k)$ for all $k\in\N$ once $(0,0)$ and $(0,-1)$ are infected. No further infections occur, as $u$ is stable and $\H_u\cup l_u$ is a translate of $\H_u$. Thus, $\alpha(0)=1$, as $l_{0}=\overline{\{(0,0)\}}$ is infinite ($l_u$ is infinite for any rational direction $u$). Turning to $u=\pi/2$, we have $\alpha(u)\ge 1$ as before, and one can check as above that $[\H_u\cup \{(0,0),(1,0)\}]=\H_u\cup l_u$ using the first and third rules. It remains to see that there does not exists $x\in\Z^2$ such that $\overline{\{x\}}$ is infinite, in order to conclude that $\alpha(u)=2$. Indeed, all rules contain at least $2$ sites in $\Z^2\setminus\H_u$, so for any $x$ we have $[\{x\}\cup\H_u]=\{x\}\cup\H_u$. Finally, once we know that $\alpha(\pi/2)=2$ and $\alpha(0)=1$, we have that the open half-circle $(-\pi/2,\pi/2)$ only contains one stable direction and it has difficulty $1$, so $\alpha\le 1$, which is the smallest possible value for a critical model: by definition, every half-circle contains a stable direction and, as we noted, only unstable directions have difficulty $0$.

The result of~\cite{Bollobas14} states that\footnote{They actually give matching bounds up to a constant factor, which requires dividing critical models into two subclasses with different logarithmic factors.}
\[p_c(n)=\frac{{(\log\log n)}^{\cO(1)}}{{(\log n)}^{1/\alpha}}.\]

\subsection{Results}

So far it has not been investigated how one could determine the difficulty $\alpha$ in practice, mainly owing to the simple definition and to the fact that for simple models such as the ones in Figure~\ref{fig:example} this is straightforward. In this paper we consider $\alpha$ from a computational perspective.

Throughout the paper, we assume that $\cU$ is described as a family of sets of pairs of integer coordinates represented in binary. Therefore the size of the input is proportional to
\begin{equation}\label{eq:def:inputsize}
	\|\cU\|:=\log D\cdot\sum_{U\in\cU}|U|,
\end{equation}
where $D$ is the ``diameter'' of $\cU$:
\begin{equation}\label{eq:def:D}
D=2\cdot\max\left\{\|x\|_\infty\;:\;x\in \bigcup_{U\in\cU}U\right\}.
\end{equation}
A further justification of the need to take $D$ into account in $\|\cU\|$ is provided in the Appendix showing that the difficulty $\alpha$ is not bounded in terms of $\sum_{U\in\cU}|U|$ only. Our first result is that $\alpha$ is computable. We prove this by giving an explicit algorithm and bounding its complexity.
\begin{thm}\label{th:decidable}
There exists an algorithm which, given a critical bootstrap percolation update family $\cU$, computes its difficulty $\alpha$.\footnote{This result is proved independently by Balister, Bollobás, Morris and Smith~\cite{Balister18b}.}
\end{thm}
\begin{rem}\label{rem:complex}
In fact, it is not hard to check that our algorithm runs in time at most
\[|\cU|^2\cdot 2^{D^2(1+o(1))}=\exp(\cO(D^2)),\]
which is in the worst case at most doubly exponential in $\|\cU\|$. This bound is as sharp as a bound in terms of $D$ only can be. Indeed, $|\cU|=e^{\cO(D^2)}$ and $|\cU|$ can be as large as $2^{D^2}$.
\end{rem}

Explicit bounds analogous to the ones derived in the proof of Theorem~\ref{th:decidable} are the only missing ingredient causing the constants appearing in the main results of~\cites{Bollobas14,Hartarsky19I} to be implicit (cf~\cite{Bollobas14}*{Lemma~6.5} and its version in~\cite{Hartarsky19I}).

Moreover, a corresponding uncomputability result in higher dimensions based on supercritical models in two dimensions has been announced by Balister, Bollobás, Morris and Smith~\cite{Balister18b} prior to our work. As that could lead one to expect, Theorem~\ref{th:decidable} is not at all automatic.

On a high level, the main idea behind our proof is that if a half-plane $\H_u$ is infected, the process restricted to the line $l_u$ is a 1-dimensional bootstrap percolation process. Owing to the bounded range of the rules and translation invariance, the final state of this process is either periodic with bounded period or finite, which two possibilities can be distinguished in a correspondingly bounded time.

On the other hand, we also prove the following negative result.
\begin{thm}\label{th:NP}
The problem of computing the difficulty $\alpha$ of a critical bootstrap percolation update family $\cU$ is NP-hard.
\end{thm}

This result is proved by a fairly technical reduction to the \textsc{Set Cover} decision problem in Section~\ref{sec:NP}. Besides the result of~\cite{Balister18b}, another reason to expect that the problem of determining $\alpha$ is hard in a sense made clear in Theorem~\ref{th:NP} is a recent parallel notion of difficulties adapted to subcritical models termed ``critical densities''. Those were introduced by the first author~\cite{Hartarsky18subcritical} and they are clearly far too complicated for one to expect to be able to compute them. From this point of view the result of Theorem~\ref{th:NP} is not unexpected.

\section{Decidability: proof of Theorem~\ref{th:decidable}}\label{sec:decidable}
In this section we provide an algorithm to compute the difficulty of a critical model. Let us stress that it is not optimized and is only meant to prove Theorem~\ref{th:decidable}.

\begin{proof}[Proof of Theorem~\ref{th:decidable}]
Fix an update family $\cU$. To start, let us see how to determine the set of stable directions in time polynomial in the size of the input $\|\cU\|$. Indeed, for each site $x$ in each rule $U$ we determines its polar coordinates $(r_x,\theta_x)=(\|x\|_2,x/\|x\|_2)\in\R_+\times S^1$. On the practical side, $r_x$ can be represented as the square root of an integer bounded by $D^2$ and $\theta_x$ can be encoded by its tangent, which is rational with numerator and denominator bounded by $D$, and one boolean indicating whether $\theta_x\in(-\pi/2,\pi/2)$. Then for each rule $U$ we take an arbitrary $x_0\in U$ and compute $\theta_x-\theta_{x_0}$ for all $x\in U$ (its tangent is still rational and its numerator and denominator are bounded by $D^2$). We determine the largest and smallest such values, $\delta_+,\delta_-$, considering differences in $\left(-\pi,\pi\right]$. Finally, the unstable interval of $U$ is $(\theta_{x_0}+\delta_++\pi/2,\theta_{x_0}-\delta_-+3\pi/2)\subset S^1$ (which is empty if $\delta_+-\delta_-\ge\pi$). The set of unstable directions is then the union of these intervals for all $U\in\cU$. In particular, the isolated stable directions and, more generally, the endpoints of the intervals of stable directions for $\cU$ are among the endpoints of the intervals for different $U$, so there are at most $2|\cU|$ of them. In order to determine this union in practice it suffices to check for each of these endpoints whether it is stable (not contained in any of the unstable intervals for other $U\in\cU$) and keep the information whether it was a left or right endpoint of the associated interval. Hence, the preliminary step of determining the (isolated) stable directions is completed in polynomial time in $\|\cU\|$. It is also not hard to verify for each of the $|\cU|$ right-endpoints whether there exists a stable direction in the half-circle starting there and whether there are finitely many of them (i.e.\ all are isolated), which allows one to decide if $\cU$ is supercritical, critical or subcritical in polynomial time.

Assuming that $\cU$ is determined to be critical, we can use~\eqref{eq:alpha:def} to compute the difficulty, $\alpha$, once we know all $\alpha(u)\in\N$ for isolated stable directions. Indeed, for each of the open semi-circles with one endpoint among those considered above, we only need to calculate the maximum of $\alpha(u)$ for isolated stable directions $u$ (if there are any non-isolated directions, we do not need to consider the semi-circle). As this can also be done in time polynomial in $\|\cU\|$, we will now fix an isolated stable direction $u$ and provide an algorithm for determining $\alpha(u)$.

We shall assume that $D$ is sufficiently large throughout the proof. Indeed, given $D$, all $U\in\cU$ are distinct subsets of ${\{-D/2,\dots,D/2\}}^2$, so there are at most $2^{2^{{(D+1)}^2}}$ possible $\cU$ and $|\cU|\le 2^{{(D+1)}^2}$. Therefore, the algorithm's asymptotic complexity is only determined by families with large values of $D$, as one can directly list the difficulties for isolated stable directions with ``small'' values of $D$ in constant time.

Recall the notation $\bar Z$ from Definition~\ref{def:diff}, which we shall use without specifying $u$, as it will be clear from the context. In order to determine $\alpha(u)$ we will use the following lemmas to bound the size of the set $Z$ in Definition~\ref{def:diff}. The first of these is a one-dimensional result which we shall reduce the problem to.
\begin{lem}\label{lem:1d:bound}
Let $\cU$ be an update family, let $u\in S^1$ be an isolated stable direction and let $A$ be a finite subset of $l_u$. Then the set $\bar A$ is either infinite or its maximal distance from $A$ is at most $D^3\cdot 2^D$.
\end{lem}
\begin{proof}
Observe that by stability of $u$ we have $\bar A\subset l_u$. Then the dynamics started from $\H_u\cup A$ can be viewed as a dynamics on $l_u$ only. Note that $l_u$ consists of integer sites on a line, so it is naturally identified with $\Z$ by the composition of a homothety and a rotation. Furthermore, we know that $u$ is an isolated stable direction and, thereby, $l_{u+\pi/2}$ (which is simply a rotation of $l_u$) contains a site $x$ in some $U\in\cU$ with $\|x\|_\infty\le D/2$ by~\eqref{eq:def:D}. Hence, the homothety ratio is between $1/D$ and $1$.

Notice that the dynamics restricted to $l_u$ is simply a $1$-dimensional bootstrap percolation process, where each rule $U\in\cU$ is replaced by $U\cap l_u$ if $U\subset (\H_u\cup l_u)$ and removed otherwise. It therefore suffices to prove the following claim, which concludes the proof.
\end{proof}

\paragraph{Claim.}
For a one-dimensional bootstrap percolation family and a finite set $A\subset \Z$, we have that $\bar A$ is either infinite or its maximal distance from $A$ is at most ${D^2}\cdot{2^D}$.

\begin{proof}
Denote $A=\{a_1,\ldots,a_n\}$ with $a_1<\cdots<a_n$. Let us denote by $P$ the property that the following three conditions hold:
\begin{itemize}
	\item $|[A]|<\infty$, $d(s,A)\le D\cdot2^{D+1}$ for all $s\in[A]$,
	\item $\max[A]-a_n\le D\cdot2^{D+1}-D$,
	\item $a_1-\min[A]\le D\cdot2^{D+1}-D$.
\end{itemize}
Let $A$ be minimal with respect to inclusion violating $P$. We next prove that $|[A]|=\infty$.

\medskip

\paragraph{Base.}
Assume that $|A|=1$, without loss of generality $A=\{0\}$. If $[A]=A$, we have nothing to prove, as $P$ clearly holds. Otherwise, assume that $x\in\Z$ becomes infected on the first step. Then, since $\{0\}$ is the only infected site initially, $\{x\}$ is a rule in the update family. However, that entails that $k.x$ becomes infected on the $k$-th iteration at the latest and, in particular, $[A]$ is infinite.

\paragraph{Step.}
Assume that $|A|>1$. Assume for a contradiction that there exists $0<i<n$ and $b\in[A]$ such that $a_{i+1}>b>a_i$ and $\min(b-a_i,a_{i+1}-b)>D\cdot2^{D+1}$. Then, by minimality of $A$, both $A'=\{a_1,\ldots,a_i\}$ and $A''=A\setminus A'$ satisfy $P$. Therefore,
\[\min[A'']-\max[A']>D\cdot2^{D+2}-2(D\cdot2^{D+1}-D)>D,\]
so that $[A]=[A']\cup [A'']$, which contradicts the existence of $b\in[A]$. Indeed, there is no site in $\Z$ such that a rule translated by it intersects both $[A']$ and $[A'']$ and by definition of the closure those do not evolve under the dynamics.

Assume next that $\max([A])>a_n+D\cdot2^{D+1}-D$ (the corresponding case for $\min([A])$ is treated identically). Then, by the pigeon-hole principle, there exist $b,c\in\Z$ with $a_n+D<b<c-D<\max([A])-2D$ such that
\[\varnothing\neq [A]\cap[b,b+D-1]=([A]\cap[c,c+D-1])-(c-b)\]
(since no infection can cross a region of size $D$ not intersecting $[A]$ to reach $\max([A])$). Therefore, $[A]\cap[b,b+D-1]$ infects a translate of itself, since the dynamics to the right of $b+D$ is not affected by infections to the left of $b$, once we fix the state of $b,\ldots, b+D-1$. Similarly to the case $|A|=1$, this is a contradiction with $|[A]|<\infty$, which concludes the proof.
\end{proof}
The next lemma is an application of the covering algorithm of~\cite{Bollobas15}. For the sake of completeness, we will include a sketch of it in the proof.
\begin{lem}\label{lem:covering:algo}
Let $\cU$ be a critical update family and $u$ be an isolated stable direction. Let $Z\subset \H_{u+\pi}$ be a set of size at most $D$. Then for every $z\in[Z]$ we have $\<z,u\>\ge -\cO(D^4)$.
\end{lem}
\begin{proof}
First, we prove the following claim.

\paragraph{Claim.}
There exists a set $\cT\supset\{u\}$ of three or four stable directions containing the origin in their convex envelope (if viewed as a subset of $\R^2$) such that for each $v\in\cT$ there exists $x\in\Z^2\cap v\R$ such that $\|x\|_\infty\le D/2$ and such that for every $v,w\in\cT$ we have $|v-w+\pi|\ge 2/D^2$.
\begin{proof}
First assume that $u+\pi$ is unstable. Let $\cT$ consist of $u$ and the stable directions, $u+\pi+\delta_+$ and $u+\pi-\delta_-$ ($\delta_\pm\in\left(0,\pi\right]$), closest to $u+\pi$ in both semi-circles with endpoint $u+\pi$ (these exist as the set of stable directions is closed). Furthermore, recalling that $\cU$ is not supercritical, there is no semi-circle of unstable directions, so $\delta_++\delta_-<\pi$. This implies that indeed $\cT$ contains $0$ in its convex envelope.

Assume that, on the contrary, $u+\pi$ is stable. Consider the semi-circle $(u,u+\pi)\subset S^1$. In it there exists a stable direction (since $\cU$ is not supercritical). If there are no unstable directions, we pick $u_-=u+\pi/2$, otherwise, we set $u_-$ to be an isolated or semi-isolated stable direction in that semi-circle. We define $u_+$ similarly in the opposite semi-circle. We set $\cT=\{u,u+\pi,u_-,u_+\}$. It is clear that $0$ is in the convex envelope of $\cT$.

In both cases $\cT$ consists of directions which are either isolated, semi-isolated or a rotation by $\pi/2$ of such a direction. Therefore, as in the proof of Lemma~\ref{lem:1d:bound}, there exists a site $x$ as in the statement of the claim.

Finally, let us bound the difference between two directions $v\neq w$ such that there exist $x\in \Z^2\cap v\R$ and $y\in\Z^2\cap w\R$ with $\max(\|x\|_\infty,\|y\|_\infty)\le D/2$. Indeed, $\det(x,y)\in\Z\setminus\{0\}$, so \[|\sin(v-w)|=\frac{|\det(x,y)|}{\|x\|_2\|y\|_2}\ge\frac{2}{D^2}\]
and therefore $|v-w|\ge 2/D^2$.
\end{proof}

We fix a set $\cT$ as in the claim. We call a $\cT$-droplet a polygon with sides perpendicular to the directions in $\cT$. Since $\cT$ contains $0$ in its convex envelope there exist $\cT$-droplets. Since the difference between consecutive directions in $\cT$ are at most $\pi-2/D^2$, we can find a $\cT$-droplet $P$ with diameter $\cO(D^3)$ containing ${[-D/2,D/2]}^2\supseteq\bigcup_{U\in\cU}U$ (e.g.\ a $\cT$-droplet circumscribed around a circle with $D$).

We can then directly apply the covering algorithm of~\cite{Bollobas15} to conclude the proof. Let us outline that algorithm in our setting. We start with a set of translates of $P$, namely $\{z+P,z\in Z\}$. At each step if two of the current droplets $P_1,P_2$ satisfy that there exists $x\in\Z^2$ such that $(P+x)\cap P_1\neq\varnothing$ and $(P+x)\cap P_2\neq\varnothing$, then we replace them by the smallest $\cT$-droplet containing their union. We repeat this as long as possible.

By Lemma~4.6 of~\cite{Bollobas15} (stating that the diameter of the smallest droplet containing two intersecting ones is at most the sum of their respective diameters) the sum of diameters of droplets increases by at most $\mathrm{diam}(P)=\cO(D^3)$. Therefore, in the final set of droplets the total diameter is $\cO(D^4)$, as the number of droplets decreases by $1$ at each step. Moreover, by Lemma~4.5 of~\cite{Bollobas15} the union of the final droplets contains $[Z]$, so the proof is complete, as each of the output droplets contains at least one site of $Z\subset\H_{-u}$.
\end{proof}

\paragraph{Algorithm.} Let us first describe an algorithm to determine $\alpha(u)$ and postpone its analysis. For each integer $k$ from $1$ to $D$ we successively perform the following operations to determine if there exists a set $Z$ of size $k$ as in Definition~\ref{def:diff}. We stop as soon as such a set is found and return the corresponding (minimal) value of $k$. For each fixed $k$ we start by choosing a set $Z_0$. The first site is $0$ and each new one $z$ is picked within distance $D^{11}\cdot 2^D$ from some of the previous ones and such that
\begin{equation}
\label{eq:z'}0\le\<z'-z,u\>=\cO(D^4)
\end{equation}
for some $z'$ among the previous ones. There are at most
\[\binom{D^{\cO(1)}\cdot 2^D}{D}=2^{D^2+o(D^2)}=\exp(\cO(D^2))\]
such choices. For each of them we successively inspect different translations $t\in\Z^2$, such that $0\le \<t,u\>=\cO(D^5)$ and \begin{equation}
\label{eq:txy}0\le \<t,(-y,x)\>< x^2+y^2,
\end{equation}
where $(-y,x)\in\Z^2$ is such that $(x,y)\in u\R$ and $x$ and $y$ are co-prime, in the (total) order given by $\<t,u\>$ starting from $t=0$. Finally, fix $Z=Z_0+t$.

For each $Z$ we run the bootstrap dynamics with initial set of infections $Z\cup\H_u$ until it either stops infecting new sites or infects a site $s$ with $\|s\|_\infty\ge D^{13}\cdot 2^D$ and $\<s,u\>=\cO(D^5)$. This can be done by checking at each step each site at distance $D^{13}\cdot2^D+D$ from the origin for each rule and repeating this for $5^D$ time steps. If the dynamics becomes stationary, we continue to the next choice of $Z$, while otherwise we return $|Z|$ for the value of $\alpha(u)$.

\paragraph{Correctness.} We now turn to proving that the algorithm does return an output and it is precisely $\alpha(u)$. The first assertion is easy. Indeed, as $u$ is an isolated stable direction, (by~\cite{Bollobas14}*{Lemma~2.8}) there exists a rule $U\in\cU$ with
\[U\subset \H_u\cup \{x\in l_u,\<x,u+\pi/2\>>0\},\]
so that adding $D$ consecutive sites on $l_u$ to $\H_u$ is enough to infect a half-line of $l_u$, only taking $U$ into account. Thus, we know that $\alpha(u)\le D$ and the algorithm will eventually check such a configuration when $k=D$, unless it has returned a smaller value, and infections will propagate to distance $D^{13}\cdot2^D$ (and in fact to infinity). Let us then prove that the output is $\alpha(u)$.

Denote by $t_j$ the values of $t$ considered by the algorithm, so that $t_0=0$. Note that by~\eqref{eq:txy} there exists a single $t\in\Z^2$ with a given value of $\<t,u\>$, so that this scalar product indeed defines a total order on the values of $t$ and we can also extend our notation to $j<0$ for convenience, though those are not examined by the algorithm. Further define $l_j:=\{s\in\Z^2,\<s,u\>=\<t_j,u\>\}$ and $Z_j=Z_0+t_j$ for some $Z_0$ considered by the algorithm, so that $l_0=l_u$ by abuse of notation.\footnote{Here we view $0$ as an element of $\Z$, possible value of $j$, while $u$ is an element of $S^1$. As we will not make reference to $l_v$ with $v=0\in S^1$, we hope that this will not lead to confusion.}

\paragraph{Claim.}
Assume that a set $Z_i$ considered by the algorithm is of size $k\le\alpha(u)$ and such that $\bar Z_j$ (recall Definition~\ref{def:diff}) is finite for all $0\le j\le i$. Then the maximal distance between a site from $\bar Z_i$ and $Z_i$ is at most $D^5\cdot2^D\max(0,\<t_i,u\>)$.

\begin{proof}
We prove the statement by induction on $i\in\Z$. For $i<0$, i.e. $\<t_i,u\><0$, then $Z_i\subset \H_u$ by~\eqref{eq:z'} and there is nothing to prove, since $\bar Z_i=\varnothing$ -- no additional infections take place. Assume the property to hold for all $t_j$ with $j\le i$. We aim prove the same for $i+1$.

Observe that for each $0< j\le i+1$ we have that
\begin{equation}\label{eq:barZiplus1}
	\bar Z_{i+1}\cap l_j\subseteq (\bar Z_{i+1-j}\cap l_0)+t_{i+1}-t_{i+1-j}.
\end{equation}
Indeed, $Z_{i+1}\cup\H_u\subseteq (Z_{i+1-j}\cup \H_u)+t_{i+1}-t_{i+1-j}$, since $Z_{i+1}=Z_{i+1-j}+t_{i+1}-t_{i+1-j}$ and $\H_u+t_{i+1}-t_{i+1-j}\supset \H_u$. Furthermore, by stability of $u$ we have that $\bar Z_{i+1}\cap l_{j}=\varnothing$ for $j>i+1$. Also, by~\eqref{eq:barZiplus1} and the induction hypothesis we have that $\bar Z_{i+1}\setminus l_0$ is at distance at most $D^5\cdot2^D\<t_i,u\>$ from $Z_{i+1}$, so we are left with proving that sites in $\bar Z_{i+1}\cap l_0$ are at distance at most $D^5\cdot2^D\<t_{i+1},u\>$ from $Z_{i+1}$.

Consider the set
\[Z'=\{z\in \bar Z_{i+1}\cap l_0, d(z,Z_{i+1})\le D+D^5\cdot 2^D\<t_i,u\>\}.\]
By the reasoning above we have that $\bar Z_{i+1}\cap l_0=Z'\cup \bar Z'$. However, by Lemma~\ref{lem:1d:bound},  $\bar Z'$ cannot be at distance more than $2^D\cdot D^3$ from $Z'$, as $\Z_{i+1}\setminus l_0$ is at distance at least $D$ from all sites in $\bar Z_{i+1}\setminus Z'$. Recalling the definition of $Z'$, we get that $\bar Z_{i+1}$ is at distance at most $D+D^3\cdot 2^D+D^5\cdot 2^D\<t_i,u\>$ and we are done. Indeed, $\<t_{i+1}-t_i,u\>\ge 1/D$, since there exists a site $x\in\Z^2\cap u\R$ with $\|x\|_\infty\le D/2$ and $\<t_{i+1}-t_i,x\>>0$ is an integer.
\end{proof}

The claim clearly implies that the algorithm cannot return a value smaller than $\alpha(u)$. In order to conclude, we need to show that when $k=\alpha(u)$ among the sets examined by the algorithm there will be a set $Z$ such that there exists $z\in\bar Z$ with $\|z\|_\infty\ge D^{13}\cdot 2^D$ and therefore the output will be $\alpha(u)$.

Consider a set $Z\subset \Z^2\setminus \H_u$ as in Definition~\ref{def:diff} of size $\alpha(u)$ (and therefore minimal). Recall that by Lemma~\ref{lem:covering:algo} every $z\in Z$ satisfies $\<z,u\>=\cO(D^4)$ (otherwise $\bar Z=[Z]$ is finite, as $\cU$ is not supercritical) and, by stability of $u$, the same holds for $\bar Z$. Let $\cP=\{x\in\R,\exists z\in Z, \<z,u\>=x\}$ and define $\bar\cP$ similarly for $\bar Z$. These are discrete subsets of $\R$. Note that by minimality of $Z$ and Lemma~\ref{lem:covering:algo}, $\cP\subset \R$ cannot have a gap of length larger than $\cO(D^4)$. Indeed, there exists $x\in \bar\cP$ such that infinitely many points of $\bar Z$ project to it and those are all in $\bar Z'$ where $Z'$ are the sites in $Z$ that project to $x'\in \cP$ such that there exist $n$ and $x_0=x,x_1,\dots, x_n=x'$ in $\cP$ with $|x_{j+1}-x_j|= \cO(D^4)$ and if $Z'\neq Z$, we obtain a contradiction with the minimality of $Z$.

Analogously, let $\cP^\perp=\{x\in\R,\exists z\in Z,\<z,(u+\pi/2)\>=x\}$ and define $\bar \cP^\perp$ similarly for $\bar Z$. We claim that its $\cP^\perp$ cannot have a gap of length larger than $\cO(D^{10}\cdot2^D)$. This time $\bar\cP^\perp$ is necessarily infinite, as only a finite number of points $z\in\Z^2$ with $\<z,u\>=\cO(D^4)$ have the same $(u+\pi/2)$-projection. Considering a set $Z'\subset Z$ inducing the corresponding distance $\cO(D^{10}\cdot2^D)$-connected component of $\cP^\perp$ and using the claim instead of Lemma~\ref{lem:covering:algo} as in the previous paragraph, we reach a contradiction with the minimality of $Z$.

Hence, all $Z$ of size $\alpha(u)$ as in Definition~\ref{def:diff} are in fact considered by the algorithm. Since such a $Z$ with infinite $\bar Z$ exists, the algorithm does indeed output $\alpha(u)$.
\end{proof}

\section{NP-hardness: proof of Theorem~\ref{th:NP}}\label{sec:NP}
In this section we prove Theorem~\ref{th:NP} by providing a reduction from \textsc{Set Cover} to \textsc{2D Critical Bootstrap Difficulty}. For the \textsc{Set Cover} problem we consider a \emph{universe} $\{1,\ldots, N\}$ and a collection $\cS$ of subsets of the universe and assume that $|\cS|\ge 4$ and $N\ge 4$. The \textsc{Set Cover} problem asks for determining the minimum cardinality of a subset of $\cS$ which covers the universe. It is one of the first NP-complete problems described by Karp~\cite{Karp72}.

We fix an instance
\[\cS=\left\{S_i\;:\;i\in\Z, 1\le i\le |\cS|\right\}.\]
Our goal is to define a critical bootstrap percolation update family whose difficulty $\alpha$ is (up to a simple transformation) the solution to \textsc{Set Cover}. Let the set of rules associated to $\cS$ be \[\cU_\cS=\{U_0,U_1\}\cup \{U_{i,j}^k\;:\;1\le i\le |\cS|,\,1\le k\le |\cS|^2,\,i,k\in\Z,\,j\in S_i\},\]
where
\begin{align*}
U_0=&\left\{(-k,0),(0,-k)\;:\;1\le k\le N|\cS|^2\right\},\\
U_1=&\left\{(+k,0),(0,-k)\;:\;1\le k\le N|\cS|^2\right\}
\end{align*}
and the rules $U_{i,j}^k$, defined as follows, share a large portion of their structure (see Figure~\ref{fig:rule}).
\begin{align*}
T=&\left\{(0,-y)\;:\;1\le y\le N\cdot|\cS|^2\right\},\\
W=&\{(x,0)\;:\;1\le x\le |\cS|^2\}\cup \{(l\cdot |\cS|,1)\;:\;1\le l\le |\cS|\}, \\
U_{i,j}^k=&T\cup\left(\left(W\cup\{(i\cdot|\cS|,2)\}\right)-(k+(N+j)\cdot |\cS|^2,0)\right).
\end{align*}

\begin{figure}
	\begin{center}
	\begin{tikzpicture}[scale=0.4]
		\def\S {4}
		\fill[pattern=north west lines] (1,0) rectangle ++(16,1);
		\foreach \x in {1,...,4}
			{\fill[pattern=north west lines] (\x*4,1) rectangle ++(1,1);}
		\fill[pattern=north west lines] (3*4,2) rectangle ++(1,1);
		\fill[fill=gray!30] (28,0) rectangle ++(1,1);

		\draw[decorate,decoration={brace,amplitude=6pt}] (1,3)--(17,3) node [black,midway,yshift=12.0] {\tiny $W\cup \{(i\cdot|\cS|,2)\}$};
		\draw[decorate,decoration={brace,amplitude=6pt}] (26,3)--(30,3) node [black,midway,yshift=12.0] {\tiny region of $j\in{[1,N]}$};

		\foreach \x/\l in {2/2,3/\dots,4/$|\cS|$,6/\dots,8/$2|\cS|$,10/\dots,12/$i|\cS|$,14/\dots,16/$|\cS|^2$,20/\dots,28/$k+(N+j)|\cS|^2$}
			\node at (\x+0.5,-0.5) {\tiny \l};

		\node[anchor=east] at (2,-0.5) {\tiny $x=1$};
		\node[anchor=east] at (0,0.4) {\tiny $y=0$};
		\node[anchor=east] at (0,1.4) {\tiny $1$};
		\node[anchor=east] at (0,2.4) {\tiny $2$};

		\draw[step=1] (0,0) grid (18,3);
		\draw[step=1] (23,0) grid ++(8,3);
	\end{tikzpicture}
	\caption{A visualisation of $(U_{i,j}^k\setminus T)+(k+(N+j)|\cS|^2,0)$; the shaded cell indicates where the origin is shifted to.}\label{fig:rule}
	\end{center}
\end{figure}
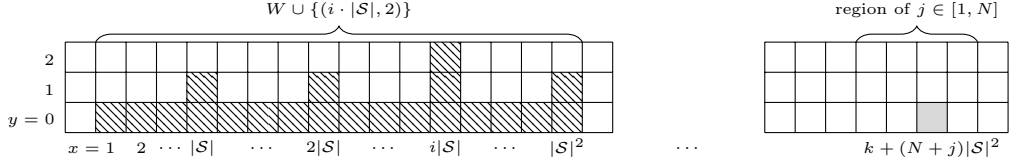

First we claim that the only isolated stable direction is $u=\pi/2$ and $[-\pi,0]$ contains the rest of the stable directions. The unstable intervals corresponding to the rules $U_0$ and $U_1$ are $(0,\pi/2)$ and $(\pi/2,\pi)$, respectively. The unstable interval of $U_{i,j}^k$ is contained in $(0,\pi/2)$ for all $i,j,k$. Thus, $\cU_\cS$ is indeed critical and $\alpha(\cU_\cS)=\alpha(u)$, so that we can focus on this direction.

Let $M\subseteq \{1,\ldots,|\cS|\}$ be an optimal solution to the \textsc{Set Cover} problem given by $\cS$ i.e.\ a set of minimal size such that
\[\bigcup_{i\in M}S_i=\{1,\ldots,N\}.\]
We first claim that, setting
\[ Z_0=W\cup \{(i\cdot |\cS|,2)\;:\;i\in M\} \]
we have $[Z_0\cup\H_u]\supset l_u$, so that
\begin{equation}
\label{eq:alpha:upper}
\alpha(u)\le |Z_0|=|W|+|M|=|\cS|^2+|\cS|+|M|.
\end{equation}
Indeed, using once each of the rules $U_{i,j}^k$ for all $i\in M$, $j\in S_i$ and $1\le k\le |\cS|^2$, one infects all sites in
\[\left[1+(N+1)\cdot|\cS|^2,(2N+1)\cdot|\cS|^2\right]\times\{0\},\]
since $M$ is a cover, and those are enough to infect $l_u$ using $U_0$ and $U_1$.

For any $Z\subseteq\Z^2$ recall the notation $\bar Z=[Z\cup\H_u]\setminus\H_u$ from Definition~\ref{def:diff}. To prove that~\eqref{eq:alpha:upper} is actually an equality, we suppose that there exists a set $Z\subset \Z^2\setminus\H_u$ for which $|\bar Z|=\infty$ and $|Z|<|Z_0|$. Fix a minimal such set $Z$.

First note that $|U_0\setminus\H_u|=N|\cS|^2$ and similarly for $U_1$. Therefore, if there exists $p\in \Z^2\setminus\H_u$ such that one of $p+U_0$ and $p+U_1$ is a subset of $Z\cup\H_u$, then $|Z|\ge N|\cS|^2>|Z_0|$ -- a contradiction. However, in order not to have $\bar Z=\varnothing$ some of the rules must be applicable to $Z\cup\H_u$ and therefore there exists $p\in \Z^2\setminus\H_u$ such that $p+W\subseteq Z$.

\begin{obs}\label{obs:W}
	For any $q\in \Z^2\setminus \{0\}$ we have $|(q+W)\setminus W|>|\cS|$.
\end{obs}

Although the verification is immediate, calling this fact an observation is deceptive, since $W$ is designed to possess this property. It follows that $p$ is unique, otherwise $|Z|> |W|+|\cS|\ge|Z_0|$ (since any minimal cover is smaller than the universe), a contradiction.

\begin{lem}\label{prop:ycoord}
	Every point $q\in \bar Z\setminus Z$ has the same $y$-coordinate as $p$.
\end{lem}
\begin{proof}
Suppose that there exists $q\in \bar Z\setminus Z$ contradicting the statement of the lemma and consider such a $q$ with minimal infection time for the process with initial set of infections $Z\cup \H_u$. Then $Z$ contains at least $|W|-|\cS|$ sites on the row of $q$, as all rules contain at least as many and by minimality of $q$. Therefore, $|Z|\ge 2(|W|-|\cS|)>|Z_0|$, a contradiction.
\end{proof}
By Lemma~\ref{prop:ycoord} and the fact that $(Z\cup\H_u)-(0,1)\subseteq (Z-(0,1))\cup\H_u$ and $(Z\cup\H_u)+(1,0)=(Z+(1,0))\cup\H_u$, we can assume without loss of generality that $p=0$.

By the minimality of $Z$ and Lemma~\ref{prop:ycoord}, the $y$-coordinate of any site in $Z$ is 0, 1, or 2. Indeed, in order to infect each of the sites $q\in\bar Z\subseteq l_u$, we use one of the rules, but those are all contained in $\{x\in\Z^2,\<x,u\>\le 2\}$, so one can remove any other sites from $Z$ without changing $\bar Z$.

\begin{lem}\label{lem:q}
There does not exist $q\in \Z^2\setminus\{0\}$ such that $q+W\subseteq\bar Z$.
\end{lem}
\begin{proof}
	Let $q$ be as in the statement of the lemma such that no other $q'+W$ becomes fully infected before $q+W$ for the process with initial infections $Z\cup\H_u$. By Lemma~\ref{prop:ycoord} we have that $q\in l_u$.

	If $|x|\ge |\cS|^2$, then by Lemma~\ref{prop:ycoord} the set $Z\setminus W$ contains at least $|W\setminus l_u|=|\cS|$ elements (with $y$-coordinate $1$), therefore $|Z|\ge |W|+|\cS|\ge |Z_0|$, a contradiction.

	Assume that $|x|<|\cS|^2$. If $l_u\cap(q+W)\setminus W\subseteq Z$, then by Observation~\ref{obs:W} we have $|Z|\ge |W|+|\cS|$ -- a contradiction. Therefore, some of the sites in $l_u\cap(q+W)\setminus W\subseteq \bar Z$ are infected by the process. However, by minimality of $q$ they can only be infected using $U_0$ or $U_1$. Yet, as soon as one can use rule $U_0$ or $U_1$ to infect a site in $l_u$, the entire $l_u$ can be infected using those rules only. Thus, removing from $Z$ every site in $Z\setminus W$ with $y$-coordinate 1 (and in particular $(q+W)\setminus(l_u\cup W)\neq\varnothing$) does not prevent the infection of infinitely many sites, which contradicts the minimality of $Z$.
\end{proof}

By Lemma~\ref{lem:q} we have that until a rule $U_0$ or $U_1$ is used the only possible infections are of the form ``$k+(N+j)|\cS|^2$ becomes infected  via rule $U_{i,j}^k$''. Therefore, all sites $(x,2)\in Z$ are either redundant (which contradicts the minimality of $Z$) or satisfy $x=i\cdot|\cS|$ with $1\le i\le |\cS|$.

	Finally, set $I=\{i\;:\;(i\cdot|\cS|,2)\in Z\}$ and
\[J=\{1,\ldots,N\}\setminus\bigcup_{i\in I}S_i.\]
Then, in order to have $|\bar Z|=\infty$, it is necessary (and sufficient) to have a sequence of $N|\cS|^2$ consecutive sites in
\[(Z\cap l_u)\cup\{(k+(N+j)|\cS|^2,0)\;:\;i\in I,1\le k\le|\cS|^2,j\in S_i\}.\]
However, such a sequence is either disjoint from the infections of the form $(k+(N+j)|\cS|^2,0)$, in which case $|Z|\ge N|\cS|^2>|Z_0|$ -- a contradiction, or disjoint from $W$. In the latter case the sequence contains at most
\[|Z|-|W|-|I|+(N-|J|)\cdot|\cS|^2< (|Z_0|-|W|)+(N-|J|)|\cS|^2\] infected sites. If $|J|\neq N$, i.e. $I$ is not a cover, the number of sites is at most $|\cS|+(N-1)|\cS|^2<N|\cS|^2$ -- a contradiction. Otherwise, $I$ is a cover and $|Z|\ge |W|+|I|\ge |Z_0|$, as $M$ is a minimal cover. This contradiction completes the proof that $\alpha(u)$ is indeed equal to $|W|+|M|=|\cS|^2+|\cS|+|M|$ as claimed.

The set $\cU_\cS$ (to which we reduced the \textsc{Set Cover} problem $\cS$) contains $|\cS|^3\sum_{S_i\in\cS} |S_i|$ rules, each of which has cardinality at most $\cO(N|\cS|^2)$, thus the reduction is indeed polynomial. This concludes the proof of Theorem~\ref{th:NP}, because $\alpha(\cU_\cS)-|\cS|^2-|\cS|$ is the size of an optimal set cover from $\cS$.

\section{Open problems}\label{sec:open}
Let us conclude with a few open questions naturally suggested by the present work. Of course, many more complexity issues arise systematically for hard problems, but let us mention the foremost ones.
\begin{ques}
	Can one find a good approximation of $\alpha$ in time polynomial of the input size $\|\cU\|$ (defined in~\eqref{eq:def:inputsize})?
\end{ques}
\begin{ques}
Are there interesting subfamilies of critical models for which the difficulty is computable in polynomial time $\|\cU\|$?
\end{ques}
\begin{ques}
	In view of Remark~\ref{rem:complex}, can one find an algorithm which computes $\alpha$ in $e^{\cO(\|\cU\|)}$ time?
\end{ques}
In the appendix we provide an example showing the $\alpha$ itself can be exponentially large in $\|\cU\|$, suggesting that one should not hope for a subexponential complexity algorithm to compute it.
\begin{ques}
	Is the \textsc{2D Critical Bootstrap Difficulty} problem in \textsc{NP} (and thus \textsc{NP}-complete)?
\end{ques}

\section*{Acknowledgments}
The authors would like to thank the organizers of ICGT 2018, Lyon, during which this project started. We also thank Rob Morris for helpful comments regarding~\cite{Balister18b}.

\appendix
\section{Relevance of the diameter}
In this appendix we provide a sequence ${(\cU_k)}_{k=2}^\infty$ of update
families such that $\sum_{U\in \cU_k}|U|$ is constant and $\alpha(\cU_k)$ is exponential in $\|\cU_k\|$. This answers a question raised during the preparation of this paper. The example gives some relevance to the questions in Section~\ref{sec:open} as well as further justifying the definition of $\|\cU\|$ in equation~\eqref{eq:def:inputsize}.  
For any integer $k\ge 2$ let $\cU_k=\{U_1,U_2\}$ with
\begin{align*}
	U_1={}&\{(0,-1),(k,0),(k-1,0)\}\\
	U_2={}&\{(0,-1),(-k,0),(-k+1,0)\}.
\end{align*}

\begin{prop}\label{prop:alphaUk}
For any integer $k\ge 2$ the update family $\cU_k$ is critical and \[\alpha(\cU_k)=k=\frac{D}{2}=\frac{1}{2}\cdot e^{\|\cU_k\|/6}.\]
\end{prop}
\begin{proof}
It is not hard to check as in the examples in Figure~\ref{fig:example} that (similarly to the Duarte model) the set of stable directions for $\cU_k$ is $[-\pi,0]\cup\{\pi/2\}$, so the model is critical. Moreover, $\alpha:=\alpha(\cU_k)=\alpha(u)$ where $u:=\pi/2$ is the only isolated stable direction.

It suffices to prove that $\alpha=k$. Consider $Z_0=\{(i,0)\ :\ i\in\{1,\dots,k\}\}$ and observe that $[Z_0\cup\H_u]=\H_u\cup l_u$. Indeed, by stability of $u$ we have $[Z_0\cup\H_u]\subseteq \H_u\cup l_u$, while using $U_1$ one can infect successively $(-i,0)$ for all $i\le 0$. Similarly, using $U_2$, one can infect $(k+i,0)$ for $i>0$.

We are thus left with proving that for any $Z\subset\Z^2$ with $|Z|<k$ we have $|\bar Z|<\infty$. Consider a minimal set $Z$ contradicting this statement.

Let $p(i,j)=(i,0)$ be the projection onto $l_u$ and let $p(Z)=\{p(z)\ :\ z\in Z\}$ be the projection of $Z$. We claim that
\begin{equation}\label{eq:barpZ}
	\overline{p(Z)}\supseteq p(\bar Z).
\end{equation}
Let $l_j=\{(i,j)\ :\ i\in\Z\}$ and let $m=\max\{j\ :\ \bar Z\cap l_m\neq\varnothing\}$. By stability of $u$ we have that $l_m\cap Z\neq\varnothing$. As $(0,-1)\in U_1\cap U_2$, we have that $p((\bar Z\setminus Z)\cap l_m)\subseteq p(\bar Z\cap l_{m-1})$. Moreover, since $U_1\cup U_2\subset\H_u\cup l_u$, we have $\bar Z\cap l_{m-1}=\overline{(Z\setminus l_m)}\cap l_{m-1}$. Therefore, if we consider $Z'=(Z\setminus l_m)\cup ((Z\cap l_m)-(0,1))$, i.e.\ we decrease the $y$-coordinates of all sites in $Z\cap l_m$ by 1, we have that
\begin{equation}\label{eq:pbarZ'}
	p(\bar Z')\supseteq p(\bar Z\cap l_m).
\end{equation}
Furthermore, as $U_1\cup U_2\subset \H_u\cup l_u$ and $Z'\cap (\H_u\cup\bigcup_{j<m}l_j)\supseteq Z\cap(\H_u\cup\bigcup_{j<m}l_j)$, we have \[\overline{Z'\cap (\H_u\cup\bigcup_{j<m}l_j)}\supseteq \overline{Z\cap (\H_u\cup\bigcup_{j<m}l_j)}.\]
Combining this with~\eqref{eq:pbarZ'}, we get that $p(\bar Z')\supseteq p(\bar Z)$. Repeating this procedure until $m=0$, we obtain~\eqref{eq:barpZ}.

By stability of $u$ we have that $\bar Z\subseteq \bigcup_{0\le j\le m} l_j$, so $\bar Z$ is infinite if and only if $p(\bar Z)$ is. Since $|p(Z)|\le|Z|$, we may replace $Z$ by $p(Z)$ and assume without loss of generality that $Z\subset l_u$. As $l_u$ identifies with $\Z$ by $(i,0)\mapsto i$, the following lemma concludes the proof.
\end{proof}

\begin{lem}
Consider the $1$-dimensional update family consisting of the rules $U_1=\{k,k-1\}$ and $U_2=\{-k,-k+1\}$. There does not exist $Z\subset \Z$ with $|Z|<k$ such that $|[Z]|=\infty$.
\end{lem}
\begin{proof}
Notice that if $z\in [Z]\setminus Z$ is used to infect another site using rule $U_1$, then either $z-k$ or $z-(k-1)$ gets infected after $z$, so $z$ is infected using rule $U_1$. Therefore, $z+k$ and $z+k-1$ are infected before $z$.

Let $Z$ be a counterexample of the statement of the lemma. Without loss of generality, we may assume that $\inf ([Z])=-\infty$. Necessarily, there exists $z\in[Z]$ with $z<\min Z-k^2$, which is infected using rule $U_1$. By the argument above, $z+k$ and $z+k-1$ are infected via rule $U_1$ (before $z$ gets infected). Iterating this argument we obtain that $X_0=\{z+k^2-k+1,\dots,z+k^2\}$ are all infected by rule $U_1$.

Let $X_i=X_0+k\cdot i$ and
\[Y_i=\{x-k\cdot i-z-k^2+k\ :\ x\in X_i,\ x\text{ is infected using $U_1$}\},\]
so that $Y_0=\{1,\dots,k\}\supseteq Y_i$ for all $i\ge0$. As in the proof of Proposition~\ref{prop:alphaUk}, one can check that $[X_0]=\Z$, so $[Z]=\Z$. Therefore, by an analogous reasoning for $U_2$, we have that all sites to the right of $Z$ are infected using rule $U_2$. Thus, $Y_{i_0}=\varnothing$ for $i_0$ sufficiently large. For any $y\in Y_{i-1}\setminus Y_i$ the site $y+k\cdot i+z+k^2-k$ is contained $\in Z$, because by definition, it does not get infected by $U_1$, and the first argument of this proof shows that it cannot be infected via $U_2$. Hence, $k=|Y_0\setminus Y_{i_0}|\le |Z|$, a contradiction.
\end{proof}
\bibliographystyle{plain}
\bibliography{Bib}

\end{document}